\documentclass[12pt]{article}
\usepackage[centertags]{amsmath}
\usepackage{amsfonts}
\usepackage{amssymb}
\usepackage{amsthm}
\usepackage{newlfont}

\newcommand{\comment}[1]{}

\newtheorem{theorem}{Theorem}
\newtheorem{lemma}{Lemma}

\newtheorem{proposition}{Proposition}

\newtheorem{example}{Example}

\newtheorem{corollary}{Corollary}
\begin{document}

\title{A general framework for perfect simulation of long memory processes}
\author{ Emilio De Santis \\
{\small Dipartimento di Matematica }\\
{\small Sapienza Universit\`a di Roma}\\
{\small \texttt{desantis@mat.uniroma1.it}} \and Mauro Piccioni \\
{\small Dipartimento di Matematica }\\
{\small Sapienza Universit\`a di Roma}\\
{\small \texttt{piccioni@mat.uniroma1.it}}} \maketitle

\begin{abstract}
In this paper a general approach for the perfect simulation of a
stationary process with at most countable state space is outlined.
The process is specified through a kernel, prescribing the
probability of each state conditional to the whole past history. We
follow the seminal paper \cite{CFF}, where sufficient conditions for
the construction of a certain perfect simulation algorithm have been
given. We generalize this approach by defining {\it backward
coalescence times} for these kind of processes; this allows us to
construct perfect simulation algorithms under weaker conditions.
Backward coalescence times are constructed in the following ways:
(i) by taking into account some a priori knowledge about the
histories that occur; (ii) by merging the algorithm in \cite{CFF}
with the classical CFTP algorithm \cite{PW}.
\end{abstract}

\medskip



\section{Introduction}\label{sec1}

Perfect simulation algorithms for stochastic processes have been
developed mostly for Markov chains, starting from the original CFTP
algorithm presented in the founding paper by Propp and Wilson
\cite{PW}. Later on, Foss and Tweedie \cite{FT1998} recognized the
fundamental role of the so-called stochastic recursive sequences for
perfect simulation. Murdoch and Green \cite{MG1998} constructed a
stochastic recursive sequence for perfect simulation when the
transition kernel satisfies a minorization condition, called the
gamma-coupler. Among the applications of perfect simulation in
recent years, we cite stochastic geometry (\cite{FFG}, \cite{M2001})
and random fields (\cite{HS00}, \cite{DP08}).

The fact that the main idea underlying stochastic recursive
sequences, and the gamma-coupler in particular, works beyond the
markovian case, is shown by the extension, due to Comets et al
\cite{CFF}, to processes with infinite memory. The aim of the
present work is to present some generalizations of their results.

In this paper we consider stochastic processes defined on
$\mathbf{Z}$ with values in an {alphabet} $G$, which is finite or
countable: thus realizations of these processes are two-sided
infinite {words}. The law of the process is obtained through a
transition kernel prescribing the probability that each letter of
the alphabet occurs in any given position of the word, conditional
to the whole history preceding it.

For notational convenience we use the following convention
concerning sequences with values in $G$:
whenever $m\leq n$ are elements of $\mathbf{Z}$ 
we define the word
$$
\mathbf{s}^{n}_{m}=(s_n,\ldots,s_m) \in G^{m-n+1}.
$$
 With an obvious extension we also allow $m=-\infty$ and $n =+
 \infty$. For brevity of notation we write $\mathbf{s}: =\mathbf{s}^{+\infty}_{-\infty}
 $. For $m=-\infty$ and $n$ finite the word will be called a {\it
 history}. Histories are elements of $G^{-\mathbf{N}^{\ast
}}$, where $\mathbf{N}^{\ast }$ is the set of positive integers. We
can concatenate any word $\mathbf{s}^{n}_{m}$ with a history
 $  \mathbf{w}^{m-1}_{-\infty}$, obtaining another history
 $$
    ( \mathbf{s}^{n}_{m}, \mathbf{w}^{m-1}_{-\infty}  )    =( s_n, \ldots, s_m , w_{m-1} , w_{m-2}, \ldots ).
 $$
The same notational conventions are used for sequences with values
in the interval $[0,1)$.

The set of histories $G^{-\mathbf{N}^{\ast}}$ is equipped with the
ultrametric distance
$$
\delta(\mathbf{s}^{-1}_{-\infty},\mathbf{t}^{-1}_{-\infty})=2^{-\inf
\{n:s_{-n} \neq t_{-n}\}},\,\,\,\,
\mathbf{s}^{-1}_{-\infty},\mathbf{t}^{-1}_{-\infty} \in
G^{-\mathbf{N}^{\ast}}   .
$$
 The corresponding Borel $\sigma$-algebra
coincides with the product $\sigma$-algebra, which is generated by
all cylinder sets.

Let us consider a kernel $p:G\times G^{-\mathbf{N}^{\ast
}}\rightarrow \left[ 0,1\right] $, which will be denoted by
$p(g|\mathbf{w}_{-\infty }^{-1})$. This means that, for any $g \in
G$, $p(g|\cdot)$ is a measurable function in $G^{-\mathbf{N}^{\ast
}}$ such that $\sum_{g\in G}p(g| \mathbf{w}_{-\infty }^{-1})=1$.

We say that a process $\left\{ X_{n},n\in \mathbf{Z}\right\} $ is
{\it compatible} with the kernel $p$, if for any $m\in \mathbf{Z}$
and $g\in G$
\begin{equation}
{P}\left( X_{m}=g|X_{m-i},i\in \mathbf{N}^{\ast }\right)
=p(g|\mathbf{X}_{-\infty }^{m-1}),\hbox{ \ a.s.} \label{specifica}
\end{equation}

For any $n \in \mathbf{N}_+$, the "one-dimensional" kernel $p$
induces an $(n+1)$-dimensional kernel $p^{(n+1)}:G^{n+1}\times
G^{-\mathbf{N}_+} \to [0,1]$ defined by
\begin{equation}\label{ricorsiva}
p^{(n+1)}(g_n,\ldots,g_0|\mathbf{w}_{-\infty
}^{-1})=\prod\limits_{k=0}^{n}p(g_{k}|g_{k-1},...,g_{0},\mathbf{w}
_{-\infty }^{-1}),\,\,\,  g_k\in G,\,\,\, k=0,1,\ldots,n, \,\,\,
\mathbf{w} _{-\infty }^{-1}\in G^{-\mathbf{N}_+}.
\end{equation}
If the process $\left\{ X_{n},n\in \mathbf{Z}\right\} $ is
compatible with $p$, then for any $m\in \mathbf{Z}$, $n\in
\mathbf{N}$ and any choice of $g_{k}\in G,k=0,...,n$
\begin{equation}
P\left( X_{m+k}=g_{k},k=0,...,n|X_{m-i},i\in \mathbf{N}^{\ast
}\right) =\prod\limits_{k=0}^{n}p(g_{k}|g_{k-1},...,g_{0},\mathbf{X}
_{-\infty }^{m-1})\hbox{ \ a.s.}  \label{ricorsiva2}
\end{equation}%
Starting from a kernel $p$, perfect simulation is aimed to construct
algorithms for sampling a compatible process $\left\{ X_{n},n\in
\mathbf{Z}\right\} $, giving at the same time sufficient conditions
for its uniqueness.

Processes of this type are known in the literature as random systems
with complete connections: for the foundations of their theory see
\cite{Ios}, whereas for a full account about the literature on these
processes we refer to the introduction in \cite{CFF}. In particular,
previously known uniqueness conditions were given in \cite{Be87} and
\cite{La00}.

We conclude this introduction by giving a plan of the paper. In
Section 2 we introduce general coupling functions in the context of
processes with infinite memory and define a general backward
coalescence time.  Its existence allows to deduce perfect simulation
algorithms for the unique stationary process compatible with the
kernel. At the end of the section the "maximal" coupling function is
described, as introduced in \cite{CFF}.

From the nature of the backward coalescence time used in \cite{CFF}
we abstract the notion of information depth, to which Section
\ref{sec3} is devoted. This is a stopping time, associated to each
instant and adapted to the past values of the random sources feeding
the coupling function, which bounds the amount of information needed
on the previous states in order to compute the state at that
instant. From an information depth we give a canonical way to
construct a backward coalescence time, provided it can be shown that
it is a.s. finite. As in \cite{CFF}, under slightly stronger
condition, there is also a regeneration structure, which can be
exploited to simulate the process for all positive times and not
just on a finite window.

In Section 4 we modify the information depth used in \cite{CFF} in
order to deal with examples like alternating renewal process and
more general processes with a random change of regime in the memory
of the past, for which the algorithm in \cite{CFF} is not
successful. The main ingredient for proving that our modification
works is the knowledge of the histories that could occur, once the
sources of randomness are produced backward in time.

Finally, in Section 5 we construct directly a backward coalescence
time, when the information depth in \cite{CFF} is always strictly
positive, denying the possibility of building from it a backward
coalescence time. In general, the construction of this modified
backward coalescence times needs the modification of the coupling
function as well. The construction also requires a positive
probability of coalescence in what we call the markovian regime:
under suitable conditions, a perfect simulation algorithm is
constructed by combining the algorithm in \cite{CFF} with the
classical CFTP in \cite{PW}. A class of examples in which these
conditions hold is finally discussed.

\section{Coupling functions, backward coalescence times and perfect simulation} \label{coupsec}

In this section we discuss some general issues involved in the
design of a perfect simulation algorithm for a process compatible
with a kernel $p:G\times G^{-\mathbf{N}^{\ast }}\rightarrow \left[
0,1\right]$ of the form described above. The first concept to be
introduced is that of coupling function. Despite the fact that in
all the examples presented throughout the paper the coupling
function is almost always the same, borrowed from \cite{CFF}, we
believe that it is useful to give an abstract definition. In
particular we choose to make explicit the dependence of the backward
coalescence time, which is defined afterwards, on the coupling
function. In fact, several backward coalescence times will be
discussed throughout the paper, different from the one used in
\cite{CFF}, for the same coupling function.

We first give the definition of admissible history, which is related
to the zeros of the kernel function. This concept will be useful in
the definitions and the results that follow.

We start with the definition of a {\it forbidden word} of the
alphabet $G$, recursively on the length. A letter $g$ of the
alphabet $G$ is forbidden if $p(g|\mathbf{w}^{-1}_{-\infty})=0$, for
any $\mathbf{w}^{-1}_{-\infty} \in G^{-\mathbf{N}^{*}}$. A word
$\mathbf{s}^{0}_{-n}=(s_{0},\ldots,s_{-n+1},s_{-n})$ of length $n+1$
is forbidden if either $\mathbf{s}^{-1}_{-n}$ is forbidden or
$$
p(s_0|\mathbf{w}^{-1}_{-\infty})=0,\,\,\, \forall
\mathbf{w}^{-1}_{-\infty} \in G^{-\mathbf{N}^{*}}:
\mathbf{w}^{-1}_{-n}=\mathbf{s}^{-1}_{-n}.
$$
For any $n \in \mathbf{N}_+$, we call $\mathcal{H}_n$ the set of
words that are not forbidden, of length $n$.

Next define the set of {\it admissible histories }
$$
\mathcal{H}= \{\mathbf{w}^{-1}_{-\infty} \in G^{-\mathbf{N}_+}:
\mathbf {w}^{-1}_{-n} \in \mathcal{H}_n, n \in \mathbf{N}_+\}.
$$
Since the set of histories $\mathbf{w}^{-1}_{-\infty} \in
G^{-\mathbf{N}^{*}}$ such that $\mathbf {w}^{-1}_{-n} =
\mathbf{s}^{-1}_{-n}$  for any fixed $\mathbf{s}^{-1}_{-n}\in
G^{n+1}$ is closed, the set $\mathcal{H}$, being an intersection of
sets of this form, is closed.

Now we prove that a process $\mathbf{X}$, which is compatible with
$p$, is such that $\mathbf{X}^{0}_{-\infty}$ belongs to
$\mathcal{H}$ with probability $1$. It is readily seen that it is
enough to prove that if $\mathbf{s}^{0}_{-n}$ is a forbidden word
for $p$, then
\begin{equation}\label{mania}
P(X_0=s_0,\ldots,X_{-n}=s_{-n})=0.
\end{equation}
In order to prove this, notice that by the definition of forbidden
word and (\ref{ricorsiva}) it is
$$
p^{(n+1)}(s_0,\ldots,s_{-n}|\mathbf{w}_{-\infty }^{-n-1})=0,
$$
for any $\mathbf{w}_{-\infty }^{-n-1} \in G^{-\mathbf{N}^{*}}$. By
integrating over $\mathbf{w}_{-\infty }^{-n-1}$ with respect to the
law of the process we conclude that (\ref{mania}) holds. With
exactly the same argument it is proved that
$\mathbf{X}^{m}_{-\infty}$ belongs to $\mathcal{H}$ w.p. $1$ as
well, for any $m \in \mathbf{Z}$.

In principle, the construction of the set of admissible histories
can be iterated, replacing in the above definitions
$G^{-\mathbf{N}_+}$ with $\mathcal{H}$. In this way the set of
forbidden words could be enlarged, and thus the set of admissible
histories could be reduced, and so on. We choose not to pursue this
kind of generalization, since the previous definition is adequate
for the examples which will be presented during the paper.

A { \it coupling function} $f$ for the kernel $p$ is a function $f:[
0,1) \times \mathcal{H}\rightarrow G$ such that for
    any $\mathbf{w}_{-\infty }^{-1} \in \mathcal{H}$ and
    any $g \in G$ the set
    $\{u\in [0,1):f(u|\mathbf{w}_{-\infty }^{-1})=g\}$ is a
    disjoint
    union of intervals $[c_i(g|\mathbf{w}_{-\infty }^{-1}),d_i(g |
    \mathbf{w}_{-\infty }^{-1}))$, $i \in \mathbf{N}_+$,
    of total length $p\left( g |
\mathbf{w}_{-\infty }^{-1}\right)$. For practical simulation
purposes we also assume that $c_i (g|\mathbf{w}_{-\infty }^{-1})$
and $d_i (g|\mathbf{w}_{-\infty }^{-1})$ can be computed by looking
at a finite portion of the history $ \mathbf{w}_{-\infty }^{-1}$.
This implies that $f$ is jointly measurable and for any $u \in
[0,1)$ the function $f(u|\cdot)$ is continuous in $\mathcal{H}$.


Since an interval $[c, d)$ is either empty or it has positive
length, the set of admissible histories is invariant under the
coupling function, in the sense that for any $u \in [0,1)$
    $$
\mathbf{w}_{-\infty }^{-1}\in \mathcal{H} \Rightarrow
(f(u|\mathbf{w}_{-\infty }^{-1}),\mathbf{w}_{-\infty }^{-1})  \in
\mathcal{H}.
$$

As a consequence of the definition, if $f $ is a coupling function
for $p$ and $U$ is a random variable uniformly distributed in
$[0,1)$, then
\begin{equation} \label{az}
P(f(U|\mathbf{w}_{-\infty }^{-1})=g)=p\left( g | \mathbf{w}_{-\infty
}^{-1}\right) ,\hbox{ for each }g \in G,\,\,\,\,\mathbf{w}_{-\infty
}^{-1}\in \mathcal{H}.
\end{equation}

Starting from $ f^{(1)}=f$, we define recursively $f^{\left(
n\right) }:\left[ 0,1\right) ^{n}\times \mathcal{H} \rightarrow G$
for any $n\in \mathbf{N}$, in the following natural way
\begin{equation}
f^{\left( n+1\right) }(u_{n},...,u_{0}|\mathbf{w}_{-\infty
}^{-1}):=f\left( u_{n}|f^{(n)}\left(
u_{n-1},...,u_{0}|\mathbf{w}_{-\infty }^{-1}\right)
,...,f^{(1)}(u_{0}|\mathbf{w}_{-\infty }^{-1}),\mathbf{w}_{-\infty
}^{-1}\right) . \label{coupling}
\end{equation}
Thus, whenever $U_{k},k=0,...,n$ are i.i.d. random variables with
the uniform distribution in $\left[ 0,1\right) $, the random vector
$\left( f^{\left( k+1\right) }(U_{k},...,U_{0}|\mathbf{w}_{-\infty
}^{-1}),k=0,...,n\right) $ has the law $p^{(n+1)}$ given in
(\ref{ricorsiva}), for any $\mathbf{w}_{-\infty }^{-1} \in
\mathcal{H}$. This means that through the iterations of the coupling
function it is possible to define, in the same probability space, a
family of processes evolving in forward time according to the given
kernel $p$, indexed by all admissible histories $\mathbf{w}_{-\infty
}^{-1} \in \mathcal{H}$.

For the implementation of a perfect simulation algorithm we require
a coupling function $f$ to admit a backward coalescence time, which
we are going to define.

Consider a sequence $\mathbf{U=}\left\{ U_{i},i\in
\mathbf{Z}\right\} $ of i.i.d. random variables, with uniform
distribution in the interval $\left[ 0,1\right) $, used as the
source of randomness for the construction of the processes of
interest. For any $m, n \in \mathbf{Z}$ with $m \leq n$ we define
the $\sigma$-algebra $\mathcal{F}^n_m=\sigma \left(
U_{i},i=m,...,n\right)$. For simplicity of notation we specify an
arbitrary reference admissible history $\mathbf{g}^{-1}_{-\infty}
\in \mathcal{H}$. We say that a measurable function $\tau_0 (
\mathbf{U}^0_{-\infty})$ with non positive values is a {\it backward
coalescence time} if it has the properties:

\begin{enumerate}
\item[H1.] $-\tau_0 $ is an a.s. finite stopping time w.r.t. the filtration $\{ \mathcal{F}
^0_{-n}: n \in \mathbf{N} \}$, i.e. $\{\tau_0 = -l\}\in \mathcal{F}
^0_{-l}$ for any $l \in \mathbf{N}$;
\item[H2.] if $\tau_0=-l$, then for any $\mathbf{w}_{-\infty }^{-\left( l+1\right) }\in \mathcal{H}$
\begin{equation}\label{lab}
f^{\left( l+1\right) }(U_{0},U_{-1},...,U_{-l}| \mathbf{w}_{-\infty
}^{-\left( l+1\right) })=f^{\left( l+1\right)
}(U_{0},U_{-1},...,U_{-l}| \mathbf{g}_{-\infty }^{-(l+1) }).
\end{equation}
\end{enumerate}

The meaning of this definition is that by pushing back the initial
time until $\tau_0$, which is computable by simulating, backward in
time, the sequence $\{U_0,U_{-1},\ldots\}$, the dependence of the
value of the coupling function at time $0$ on the history prior to
time $\tau_0$ vanishes: only the dependence on
$\{U_0,U_{-1},\ldots,U_{\tau_0}\}$ remains. It is readily seen that
property H2 remains true for $\tau_0>-l$.

Likewise we can repeat the same construction for any $n \in
\mathbf{Z}$, defining
\begin{equation}\label{shift}
\tau_n(\mathbf{U}^n_{-\infty}):=n+\tau_0(\mathbf{U}^n_{-\infty}),
\forall n \in \mathbf{Z}.
\end{equation}
If $\tau_0 $ is a.s. finite, $\tau_n$ is finite as well, by
translation invariance. For further use, for $m\leq n$, we also
define
\begin{equation}\label{tauemmenne}
\tau[m,n]=\inf \{\tau_m,\tau_{m+1},\ldots,\tau_{n-1},\tau_{n}\},
\,\,\, m \leq n.
\end{equation}

If $\tau_0$ is a backward coalescence time we define the process
$\mathbf{X}=\{X_n,\,\, n \in \mathbf{Z}\}$ as
\begin{equation}\label{x}
    X_n=\sum_{l \in \mathbf{N}}f^{\left( l+1\right) }(U_{n},U_{n-1},...,U_{n-l}| \mathbf{g}_{-\infty
}^{n-l -1 })\mathbf{1}_{\{    \tau_n =n-l \}}(
\mathbf{U}^n_{-\infty} ).
\end{equation}
Notice that the definition does not depend on the choice of the
reference history $\mathbf{g}^{-1}_{-\infty} \in \mathcal{H}$.
\begin{proposition}\label{perfetto}
   If $\tau_0$ is a backward coalescence time, then the process $\mathbf{X}$
    is stationary and it is the unique process compatible with the kernel $p$.
\end{proposition}
\begin{proof}
 The stationarity of $\mathbf{X}$ is guaranteed by construction. Let
us proceed to prove that it is compatible with $p$. By stationarity
it is enough to prove (\ref{specifica}) for $m=0$. Given the
realization $\mathbf{U}^{0}_{-\infty}$ define the non empty random
subsets of $\mathcal{H}$
\begin{equation*}
\begin{array}{c}
  I_{n+1}(U_0,U_{-1},\ldots,U_{-n})=   \\
      \{(f^{\left( n+1\right)
}(U_{0},...,U_{-n}|\mathbf{w}_{-\infty }^{-(n+1)}),\ldots, f^{\left(
1\right) }(U_{-n}|\mathbf{w}_{-\infty
}^{-(n+1)}),\mathbf{w}_{-\infty }^{-(n+1)}):\mathbf{w}_{-\infty
}^{-(n+1)}\in \mathcal{H} \},\\
\end{array}
\end{equation*}
for $n \in \mathbf{N}$, made of admissible histories in
$\mathcal{H}$, obtained by varying in all possible ways the initial
history prior to time $-n$, and then applying the coupling function
with the fixed values $U_{-n}, \ldots, U_0$, until time zero.

The sequence $I_{n+1}(U_0,U_{-1},\ldots,U_{-n})$ is non increasing
in $n \in \mathbf{N}$. Moreover each element in
$I_{-\tau[-n,0]+1}(U_0,U_{-1},\ldots,U_{\tau[-n,0]})$ has
ultrametric distance from $\mathbf {X}^{0}_{-\infty}$ which does not
exceed $2^{-n}$. Therefore $\mathbf {X}^{0}_{-\infty}$ belongs to
the closure of
$$
I_1(U_0)=\{(f(U_0|\mathbf {w}^{-1}_{-\infty}),\mathbf
{w}^{-1}_{-\infty}): \mathbf{w}_{-\infty }^{-1}\in \mathcal{H}\}.
$$
But the continuity of $f(u|\cdot)$ implies that $I_1(U_0)$ is
closed, hence $\mathbf {X}^{0}_{-\infty} \in I_1(U_0)$, meaning that
$\mathbf{X}^{-1}_{-\infty} \in \mathcal{H}$ and
$$
X_0=f(U_0|\mathbf{X}^{-1}_{-\infty}).
$$
Since $U_0$ is independent of $\cal{F}^{-1}_{-\infty}$ and $X_{-n}$
is measurable w.r.t. this $\sigma$-algebra, for any $n \in
\mathbf{N}_+$, this implies that $\mathbf{X}$ satisfies
(\ref{specifica}) with $m=0$. The proof of uniqueness is essentially
the same as in \cite{CFF} p. 935. It is a consequence of the fact
that the tail probability $P( \tau[0,n] \leq  -i )$, gives an upper
bound on the variation distance between two distributions of the
form (\ref{ricorsiva}) indexed by any two initial histories in
$\mathcal{H}$ which differ only before time $-i$. By a.s. finiteness
of $\tau[0,n]$, this tail probability goes to zero as $ i \to
\infty$.
\end{proof}

The construction of the process $\mathbf{X}$ yields a perfect
simulation algorithm on a finite window $[m,n]$, obtained first by a
backward inspection of the sequence $(U_{n-i}, i \in \mathbf{N})$ in
order to locate the stopping time $\tau[m,n]$ defined in
(\ref{tauemmenne}), and then by a recursive computation of the
coupling function started from the reference initial history
$\mathbf{g}^{\tau[m,n]-1}_{-\infty}$. In general we cannot say that
all the intermediate values of $\mathbf{X}$ prior to time $m$ are
identified during this computation, but this will happen for the
kind of backward coalescence times considered in the next section.

Mimicking the proof of Proposition \ref{perfetto} we can construct
the process $\mathbf{X}$ also when the set of admissible histories
$\mathcal{H}$ is replaced by a possibly smaller subset
$\mathcal{H}'$ in the definition of the coupling function $f$ and
the backward coalescence time $\tau_0$.

\begin{proposition}\label{confetto}
Suppose that:
\begin{itemize}
    \item [1.] $\mathcal{H}'$ is invariant under the coupling function
    and under the cut operator sending $\mathbf{w}^{-1}_{-\infty}$
    into $\mathbf{w}^{-2}_{-\infty}$;
    \item [2.] for any $u \in [0,1)$, the function
    $f(u|\cdot )$ is continuous in
    $ \mathcal{H}'$;
    \item [3.] $\mathbf{X}\in \mathcal{H}'$ with probability $1$;
    \item [4.] any $P$ compatible with the kernel $p$ gives probability
    $1$ to $\mathcal{H}'$.
\end{itemize}
    Then, if $\tau_0$ satisfies the assumptions H1 and H2 (with $\mathcal{H}$
    replaced by $\mathcal{H}'$), the process $\mathbf{X}$
    is stationary and it is the unique process compatible with the kernel $p$.
\end{proposition}

Finally we present the construction of the "maximal" coupling
function introduced in \cite{CFF}, modified by taking into account
only the trajectories in the admissible set of histories
$\mathcal{H}$. In order to present this coupling function some
relevant quantities have to be defined. First define $a_{k}:G\times
\mathcal{H}\rightarrow \left[ 0,1\right] $ as
\begin{equation*}
a_{0}(g)=\inf \left\{
p(g|\mathbf{z}^{-1}_{-\infty}):\mathbf{z}^{-1}_{-\infty}\in
\mathcal{H} \right \}, \,\,\,\, g \in G ,
\end{equation*}
\begin{equation}\label{conv1}
a_{k}(g|\mathbf{w}_{-k}^{-1})=\inf \left\{
p(g|\mathbf{z}^{-1}_{-\infty}):\mathbf{z}^{-1}_{-\infty}\in
\mathcal{H}, \mathbf{z}^{-1}_{-k}=\mathbf{w}^{-1}_{-k} \right\} ,
\,\,\,\, g \in G , \,\,\,\, \mathbf{w}_{-\infty}^{-1}
 \in \mathcal{H} ,
\end{equation}
and the increments $b_{k}:G\times \mathcal{H}\rightarrow \left[ 0,1%
\right] $ defined as
\begin{equation*}
b_{k}\left( g|\mathbf{w}_{-k}^{-1}\right)
=a_{k}(g|\mathbf{w}_{-k}^{-1})-a_{k-1}(g|\mathbf{w}_{-\left(
k-1\right) }^{-1}),
\end{equation*}
for any $k\in \mathbf{N}$, with $a_{-1} \equiv 0$ .

In order to define the maximal coupling function we need to assume
that, for any $g \in G$, the function $\mathbf{w}^{-1}_{-\infty} \in
\mathcal{H} \mapsto p(g|\mathbf{w}^{-1}_{-k})$ is continuous, i.e.
\begin{equation}\label{converge}
\sum_{k=1}^n b_k(g|\mathbf{w}^{-1}_{-k})=a_n(g,\mathbf
{w}^{-1}_{-n})\uparrow p(g|\mathbf{w}_{-\infty}^{-1}), \,\,\,
\forall g \in G, \,\,\,\mathbf{w}_{-\infty}^{-1} \in \mathcal{H}.
\end{equation}
Also define, for any $\mathbf{w}_{-\infty}^{-1} \in \mathcal{H}$ and
$k \in \mathbf{N}$
\begin{equation}\label{antonio}
a_{k}\left( \mathbf{w}_{-k}^{-1}\right) : =\sum_{g\in G}a_{k}\left(
g|\mathbf{w}_{-k}^{-1}\right).
\end{equation}
It is easily proved that (\ref{converge}) is equivalent to
\begin{equation}\label{noia}
a_{k}\left( \mathbf{w}_{-k}^{-1}\right)\uparrow 1,\,\,\,  \forall
\mathbf{w}^{-1}_{-\infty}\in \mathcal{H}.
\end{equation}
Next we partition the interval $[a_{k-1}( \mathbf{w}_{-k+1}^{-1}),
a_{k}( \mathbf{w}_{-k}^{-1}))$  in subintervals
$B_k(g|\mathbf{w}_{-k}^{-1})$ of length $b_{k}\left(
g|\mathbf{w}_{-k}^{-1}\right)$ (if this value is positive), varying
$g \in G$, for any $k \in \mathbf{N}$: in the union
$\cup_{k=0}^{\infty}B_k(g|\mathbf{w}_{-k}^{-1})$  the function
$f(u|\mathbf{w}^{-1}_{-\infty})$ takes the value $g$. Any function
of this form will be called a {\it maximal coupling function}.

Also in the definition of the maximal coupling function it is
possible to replace the set $\mathcal{H}$ with a smaller set
$\mathcal{H}'$, provided the assumptions of Proposition
\ref{confetto} are satisfied. We will see a particular example in
the sequel.

\section{Backward coalescence times constructed from information depths} \label{sec3}

In this section we present a particular class of backward
coalescence times specified through a two-stage procedure. This
concept is inspired by the particular construction of the backward
coalescence time presented in \cite{CFF}.


An {\it information depth} $K_0=K_0(\mathbf{U}^{0}_{-\infty})$ for
the coupling function $f$ is an a.s. finite stopping time w.r.t. the
filtration $\{\mathcal{F}^0_{-n}; n \in \mathbf{N} \}$ with the
property that $K_0= m$ implies
\begin{equation*}\begin{array}{c}
f^{(m+1)} ( U_0 , U_{-1} , \dots , U_{-m}|\mathbf{w}^{-(m+1)}_{-
\infty} ) \\ =f ( U_{0}|f^{(m)}( U_{-1},...,U_{-m}|
\mathbf{w}^{-(m+1)}_{- \infty} ) ,...,f^{(1)}(U_{-m}|
\mathbf{w}^{-(m+1)}_{- \infty}),\mathbf{w}^{-(m+1)}_{- \infty})
                             \end{array}
\end{equation*}
\begin{equation}\label{Kappa}
= f ( U_{0}|f^{(m)}( U_{-1},...,U_{-m}| \mathbf{w}^{-(m+1)}_{-
\infty} ) ,...,f^{(1)}(U_{-m}| \mathbf{w}^{-(m+1)}_{-
\infty}),\mathbf{g}^{-(m+1)}_{- \infty}),
\end{equation}
for any $m \in \mathbf{N}$ and any $\mathbf{w}^{-(m+1)}_{-
\infty}\in \mathcal{H}$, such that
\begin{equation}\label{nonso}
(f^{(m)}( U_{-1},...,U_{-m}| \mathbf{w}^{-(m+1)}_{- \infty} )
,...,f^{(1)}(U_{-m}| \mathbf{w}^{-(m+1)}_{-
\infty}),\mathbf{g}^{-(m+1)}_{- \infty}) \in \mathcal{H}.
\end{equation}
It is checked that when this is fulfilled it remains true for any $m
> K_0$. In fact the set of equalities (\ref{Kappa}) which have to be
checked, for a fixed $\mathbf{w}^{-1}_{- \infty} \in \mathcal{H}$,
is reduced as $m$ grows. We recall that
$\mathbf{g}^{-1}_{-\infty}\in \mathcal{H}$ is arbitrary, so if $m
\geq K_0$, the dependence of the coupling function $f^{(m+1)} (U_0,
U_{-1},\dots,U_{-m}|\mathbf{w}^{-(m+1)}_{- \infty} )$, on the
history $\mathbf{w}^{-(m+1)}_{- \infty} \in \mathcal{H}$ prior to
time $-m$, is due only to the states computed in the subsequent
interval $[-m,-1]$. For $m=0$ property (\ref{Kappa}) means that
$f(U_0|\mathbf{w}^{-1}_{- \infty})$ is constant w.r.t.
$\mathbf{w}^{-1}_{- \infty}\in \mathcal{H}$.

By comparing (\ref{Kappa}) with (\ref{lab}) it is seen that $K_0$ is
not necessarily the negative of a backward coalescence time. In
general, to eliminate completely the dependence on
$\mathbf{w}^{-(m+1)}_{- \infty} \in \mathcal{H}$ in (\ref{Kappa}), a
larger value of $m$ has to be expected. In order to construct a
backward coalescence time we define the sequence
\begin{equation} \label{ruo}
\mathbf{K}= \{K_j = K_0 ( \mathbf{U}^{j}_{-\infty}) ,j\in
\mathbf{Z}\},
\end{equation}
of information depths at all times. Next introduce the random
variable
\begin{equation}\label{ruo2}
\tau_0^{\mathbf{K}} (\mathbf{U}^0_{-\infty})=\sup \left\{ s \leq
0:K_{j}\leq j-s, s\leq j \leq 0 \right\}.
\end{equation}
Notice that, differently from $K_0$, $\tau_0^{\mathbf{K}}$ takes
negative values: indeed, by definition $\tau_0^{\mathbf{K}}\leq
-K_0$. The random variable $\tau_0^{\mathbf{K}}$ is a candidate for
a backward coalescence time; in fact the following result holds.
\begin{proposition}\label{acca}
If $\tau_0^{\mathbf{K}}$ is a.s. finite it satisfies properties H1
and H2.
\end{proposition}
\begin{proof} Let us observe that, for any $m \in \mathbf{N}$,
\begin{equation} \label{cappa}
\{-\tau_0^{\mathbf{K}}\leq m\}=\cup_{i=0}^m F_i,
\end{equation}
where
\begin{equation}\label{nome}
F_i=\{K_{-i}=0, K_{-i+1}\leq 1, \ldots, K_{-1}\leq i-1,K_0 \leq i\}.
\end{equation}
Since $F_i \in \mathcal{F}^0_{-i}$, H1 is proved.

Now assume that $F_i$ is realized, for some $i \in \mathbf{N}$. From
$K_{-i}=0$ it is obtained that $F_i$ implies
\begin{equation}\label{primopasso}
f(U_{-i}|\mathbf{w}^{-(i+1)}_{-\infty})=f(U_{-i}|\mathbf{g}^{-(i+1)}_{-\infty}),
\end{equation}
for any $\mathbf{w}^{-(i+1)}_{-\infty} \in \mathcal{H}$, and thus
$$
f^{(2)}(U_{-i+1},U_{-i}|\mathbf{g}^{-(i+1)}_{-\infty})=f(U_{-i+1}|f(U_{-i}|\mathbf{g}^{-(i+1)}_{-\infty}),
\mathbf{g}^{-(i+1)}_{-\infty})=f(U_{-i+1}|f(U_{-i}|\mathbf{w}^{-(i+1)}_{-\infty}),
\mathbf{g}^{-(i+1)}_{-\infty})
$$
for any $\mathbf{w}^{-1}_{-\infty} \in \mathcal{H}$. Using
$K_{-i+1}\leq 1$ and (\ref{Kappa}) we obtain that
$$
f(U_{-i+1}|f(U_{-i}|\mathbf{w}^{-(i+1)}_{-\infty}),
\mathbf{g}^{-(i+1)}_{-\infty})=f(U_{-i+1}|f(U_{-i}|\mathbf{w}^{-(i+1)}_{-\infty}),
\mathbf{w}^{-(i+1)}_{-\infty})=f^{(2)}(U_{-i+1},U_{-i}|\mathbf{w}^{-(i+1)}_{-\infty}).
$$
By induction, using the same argument, it is obtained that
\begin{equation}\label{libra}
f^{(i)}(U_0,\ldots,U_{-i+1},U_{-i}|\mathbf{w}^{-(i+1)}_{-\infty})=f^{(i)}(U_0,\ldots,
U_{-i+1},U_{-i}|\mathbf{g}^{-(i+1)}_{-\infty}),
\end{equation}
for any $\mathbf{w}^{-(i+1)}_{-\infty} \in \mathcal{H}$. Since
$-\tau_0^{\mathbf{K}}\leq m$ means that $F_i$ is realized for some
$i \in [0,m]$ and the property (\ref{libra}) is preserved for values
larger than $i$, we get that H2 is fulfilled, too.
\end{proof}

Under the assumption of Proposition \ref{acca} we can define the
process $\mathbf{X}$ through (\ref{lab}) and (\ref{x}), with the
shifted backward coalescence times defined by
\begin{equation}\label{zaq}
   \tau_n=\tau_n^{\mathbf{K}}=\sup \{s \leq n: K_j\leq j-s, s \leq j
\leq n \},
\end{equation}
for any $n \in \mathbf{Z}$. Since $m \in [\tau_n^{\mathbf{K}} ,n ]$
implies $ \tau_m^{\mathbf{K}} \geq \tau_n^{\mathbf{K}}$, by starting
the forward simulation from time $\tau^{\mathbf{K}}_n$, it is
possible to recover all the values $X_m$, with
$m=\tau^{\mathbf{K}}_n,\ldots,n$, through the iteration of the
coupling function with the (arbitrary) initial history $\mathbf
{g}^{\tau^{\mathbf{K}}_n-1}_{-\infty}$.

\comment{it is apparent the structure of the simulation algorithm.
It consists in scanning backward in time the uniform random
variables $ U_n , U_{n-1}, \ldots $ until the stopping time $
\tau_n$. Then $X_n = h_{n -\tau_n +1} (U_n , \ldots, U_{\tau_n})$ is
computed from the following recursion
\begin{equation}\label{recurs1}
 X_i = f ( U_i | X_{i-1}, \ldots , X_{\tau_n} ,  \mathbf{ v
}_{-\infty}^{ \tau_n -1} ), \,\,\, i =  \tau_n+1 , \ldots , n ,
\end{equation}
started from
\begin{equation}\label{recurs2}
X_{\tau_n} = f ( U_{\tau_n}| \mathbf{ v }_{-\infty}^{ \tau_n -1})
\end{equation}
where $ \mathbf{z}$ is arbitrary. If the window $ X_m , \ldots ,
X_n$ needs to be simulated the same algorithm has to be started from
the stopping time $ \min \{ \tau_m , \ldots, \tau_n\}$. Notice that
the property (\ref{algun}) has to be used to avoid double
computations of the same states. ===========}

Next we introduce a property which is stronger  than the a.s.
finiteness of $\tau_0^{\mathbf{K}}$, but easier to verify; this is
again suggested by \cite{CFF}. In order to introduce it, define the
event
$$
R^{\mathbf{K}}=\{\tau_n^{\mathbf{K}} \geq 0, \,\,\, \forall n \in
\mathbf{N}\}=\{\tau^{\mathbf{K}}[0,\infty]=0\},
$$
where $\tau^{\mathbf{K}}[0,\infty] = \inf \{ \tau^{\mathbf{K}}_i : i
\geq 0\}$. $R^{\mathbf{K}}$ belongs to the $\sigma$-algebra
$\mathcal{F} ^{+\infty}_0=\sigma(\cup_{n \in \mathbb{N}}\mathcal{F}
^n_0)$. When $R^{\mathbf{K}}$ is realized the iteration of the
coupling function, started at time $0$ from the arbitrarily chosen
history $\mathbf{g}^{-1}_{- \infty}\in \mathcal{H}$, produces $X_n$
for any $n \in \mathbf{N}$.

It is convenient to rephrase the event $R^{\mathbf{K}}$ in terms of
the information depths $\{K_n, n \in \mathbf{N}\}$. This is done in
the following simple lemma, which is essentially borrowed from
\cite{CFF}.
\begin{lemma}\label{sigaretta}
The following identity holds
\begin{equation}\label{taucappa}
  R^{\mathbf{K}}=\{K_n \leq n, n \in \mathbf{N}\}.
\end{equation}
\end{lemma}
\begin{proof}
First we prove that $R^{\mathbf{K}}$ is included in the r.h.s. of
(\ref{taucappa}). Observe that, for any $n \in \mathbf{N}$, it is
$K_j\leq j-\tau_n^{\mathbf{K}}$ for $j \in [\tau_n^{\mathbf{K}},n]$;
in particular, for $j=n$, $K_n \leq n-\tau_n^{\mathbf{K}}$. This
does not exceed $n$ provided $\tau_n^{\mathbf{K}}\geq 0$, which
proves the promised inclusion. For the converse inclusion the
argument is the following. If $K_n \leq n$ for any $n \in
\mathbf{N}$, then it is seen from (\ref{zaq}) that $0$ belongs to
the set whose supremum is precisely $\tau_n^{\mathbf{K}}$. Thus
$\tau_n^{\mathbf{K}}\geq 0$, for any $n \in \mathbf{N}$.
\end{proof}
\begin{proposition}
\label{rinnovo} If $P (R^{\mathbf{K}})>0$ then
$\tau^{\mathbf{K}}[0,\infty]$ is finite a.s.
\end{proposition}
\begin{proof} Define the sequence of events $\{R_m^{\mathbf{K}}, m \in
\mathbf{Z}\}$ as
$$
R_m^{\mathbf{K}}=\{\tau_{m+n}^{\mathbf{K}}\geq m,  n \in
\mathbf{N}\}=\{K_{m+l}\leq l, l \in \mathbf{N}\}.
$$
Without loss of generality, working with the canonical realization
of the process $\mathbf{U}$, we can identify $R_m^{\mathbf{K}}$, for
any $m \in \mathbf{Z}$, as an event of the sequence space
$[0,1)^{\mathbf{Z}}$ . Then $R_m^{\mathbf{K}} = \sigma^m
(R^{\mathbf{K}})$, where $\sigma$ is the forward unit shift on the
sequence space $[0,1)^{\mathbf{Z}}$. By using the ergodic theorem it
is deduced that there exists an a.s. finite $Y_0\leq 0$ such that
$R_{Y_{0}}^{\mathbf{K}}$ is realized. Since $0
\geq\tau^{\mathbf{K}}[0,\infty]\geq
\tau^{\mathbf{K}}[Y_0,\infty]=Y_0$, the former is a.s. finite.
\end{proof}

If $P(R^{\mathbf{K}})>0$, along the same lines of the proof of
Corollary 4.3 in \cite{CFF} , it is possible to prove also that
$\sum_{k \in \mathbf{Z}} \delta_{1_{R_k}} $ is a stationary renewal
process.

In principle, by starting the computation of the coupling function
at time $Y_0$ (from the arbitrarily chosen history
$\mathbf{g}^{Y_0-1}_{-\infty}$), we can construct the process $X_n$
for all times $n\geq Y_0$. However notice that,  $-Y_0$ being a
stopping time w.r.t. the filtration $\{\mathcal{F}^{+\infty}_{-m}; m
\in \mathbf{N} \}$, it is not accessible by simulation.


We finally describe the construction in \cite{CFF} of an information
depth for the maximal coupling function given in the previous
section, with the adjustment needed to take into account its
dependence on the admissible histories. Let us define
\begin{equation*}
a_{k}=\inf \left\{ a_{k}\left( \mathbf{w}_{-k}^{-1}\right)
:\mathbf{w}_{-\infty}^{-1}\in \mathcal{H}\right\}, \,\,\,\, k \in
\mathbf{N}_+,
\end{equation*}
where $a_{k}\left( \mathbf{w}_{-k}^{-1}\right)$ has been defined in
(\ref{antonio}), and assume that the non decreasing sequence
$\left\{ a_{k}, k\in \mathbf{N} \right\} $ tends to $1$, as
$k\rightarrow \infty $. This guarantees that $\left\{
a_{k}(g|\mathbf {w}_{-k}^{-1}),k \in \mathbf{N}\right\} $ converges
to $p(g|\mathbf{w}^{-1}_{-\infty})$, for any
$\mathbf{w}^{-1}_{-\infty}\in \mathcal{H}$ and any $g\in G$, as
$k\rightarrow \infty $: by consequence the maximal coupling function
$f(\cdot|\mathbf{w}^{-1}_{-\infty})$ is well defined for any
$\mathbf{w}^{-1}_{-\infty}\in \mathcal{H}$. Next define $K:\left[
0,1\right) \rightarrow \mathbf{N}$ as
\begin{equation}
K(u )    = \sum_{k =0}^{\infty} k\mathbf{1}_{[ a_{ k-1} , a_k )} (u)
= \inf \{ k \in \mathbf{N} : a_k >u\} \label{kappa}
\end{equation}
where $ a_{-1} =0 $. Since $a_k \leq a_k(\mathbf{w}^{-1}_{-k})$, for
any $\mathbf{w}^{-1}_{-\infty} \in \mathcal{H}$, if $K(u)=k$, the
coupling function $f(u|\mathbf{w}^{-1}_{-\infty})$ is computable by
knowing only $\mathbf{w}^{-1}_{-k}$. Thus if $a_{k}\uparrow 1$, as
$k\rightarrow \infty $, condition (\ref{Kappa}) holds for $K_0 =K
(U_0)$, hence $K(U_0)$ is an information depth.

Sufficient conditions on the sequence $\left\{ a_{k},k\in
\mathbf{N}\right\} $ which guarantee that the stopping times
$\tau_0^{\mathbf{K}} $ and $\tau^{\mathbf{K}}[0,\infty]$ (with
$K_j=K(U_j)$), are finite a.s., can be found in \cite{CFF}. We
summarize the main results in the following proposition.
\begin{proposition}
\label{th0} Let $\left\{ U_{i},i\in -\mathbf{N}\right\} $ be a
sequence of
i.i.d. random variables, uniformly distributed in the interval $\left[ 0,1%
\right) $, and let $\left\{ a_{k}\in [0,1],k\in \mathbf{N}\right\}$
be a sequence increasing to $1$. Define $K$ as in (\ref{kappa}).
\begin{itemize}
    \item [(a)]If $\sum_{k=0}^{\infty
    }\prod\limits_{j=0}^{k}a_{j}=\infty$
(which implies $a_{k}\uparrow 1$, as $k\rightarrow \infty $), then
\begin{equation*}\label{ruo3}
\tau_0^{\mathbf{K}}=\sup \left\{ s \leq 0:K (U_j)\leq j-s, s\leq j
\leq 0 \right\}>-\infty,\hbox{ a.s.}
\end{equation*}
    \item [(b)]If $\prod\limits_{j=0}^{\infty}a_{j}>0$,
then $P(R_{\mathbf{K}})>0$, thus
$$\tau^{\mathbf{K}}[0,\infty]=\sup \left\{ s \leq 0:K (U_j)\leq j-s, s\leq j
\right\}>-\infty,\hbox{ a.s.}$$
\end{itemize}
\end{proposition}



\section{An information depth depending on the whole past} \label{sec2}

In order to motivate the present section we start with a class of
examples for which the sufficient conditions of Proposition
\ref{th0} appear to be rather restrictive.

\medskip

\begin{example}\label{laburina}
Assume $G=\{-1,1\}$, and
consider a kernel  $p:G\times G^{- \mathbf{N}^{\ast }}\rightarrow
\left[ 0,1\right] $ of the following form
\begin{equation*}
p\left( j|\mathbf{w}_{-\infty }^{-1}\right) =\left\{
\begin{array}{cc}
p_k(w_{-1},j), & \hbox{\ if }           w_{-1}=\dots =w_{-k},\,\,\,
w_{-k-1}=-w_{-1}, \\
p_{\infty}(w_{-1},j), & \hbox{\ if } w_{-n}=w_{-1},    \,\,\,    n
\geq 1,
\end{array}
\right.
\end{equation*}
where
$$
P_k=\{ p_k(i,j): i,j \in \{-1,+1\}\}
$$
is a stochastic matrix for any $k\in \mathbf{N}_+\cup \{\infty\}$.
The value $p_k(i,j)$ gives the conditional probability that the next
state is equal to $j$ given that the previous $k$ states are all
equal to the current state $i$, and the $(k+1)$-th equals $-i$ (when
$k< \infty$). Since, after the first change of sign, further
information about the past is not relevant, it appears that
$p_k(i,i), i \in \{-1,+1\}$ represents the survival rates of an
alternating renewal process (see \cite{BB} pp. 32-35).

We assume that there exists $\epsilon>0$ such that, for $i=-1,1$, it
holds
\begin{equation}\label{unifrate}
\epsilon \leq p_h(i,i) \leq 1-\epsilon, \,\,\,  h \in
\mathbf{N}_+\cup \{\infty\}.
\end{equation}

In this case obviously $\mathcal{H}=G^{-\mathbf{N}_+}$; however it
is possible to consider, in the definition of the maximal coupling
function, the smaller set of histories
$$
\mathcal{H}'=\{\mathbf{w}^{-1}_{-\infty}: \sum_{i=1}^{\infty}
\delta_{w_{-i},1}=+\infty, \sum_{i=1}^{\infty}
\delta_{w_{-i},-1}=+\infty\}.
$$
The assumptions of Proposition \ref{confetto} are rather trivially
checked. 1. holds since the occurrence of $\mathbf{w}^{-1}_{-\infty}
\in \mathcal{H}$ depends on the tail of the sequence
$\mathbf{w}^{-1}_{-\infty}$. 2. is due to the fact that $f(u|\cdot
)$ is constant on sufficiently small balls in $\mathcal{H}'$.
Finally 3. and 4. are ensured by (\ref{unifrate}). \comment{ By the
assumption of continuity
$$
\lim_{k \to \infty}p_k(i,j)=p_{\infty}(i,j), \,\,\, i,j \in
\{-1,1\}.
$$}

Moreover
$$
a_0(-1)=\min\{\inf_{h\in \mathbf{N}_+}p_h(-1,-1),1-\sup_{h \in
\mathbf{N}_+}p_h(1,1)\}\geq \epsilon,
$$
$$ a_0(1)=  a_0 - a_0(-1) =\min\{1-\sup_{h
\in \mathbf{N}_+}p_h(-1,-1),\inf_{h\in \mathbf{N}_+}p_h(1,1)\}\geq
\epsilon.
$$
As a consequence
\begin{equation}\label{stopper}
\tau_0(\mathbf{U}_{-\infty}^{0})=\sup \{n\in -\mathbf{N}:
1_{[a_0(-1),a_0)}(U_n)1_{[0,a_0(-1))}(U_{n-1})+1_{[a_0(-1),a_0)}(U_{n-1})1_{[0,a_0(-1))}(U_{n})=1\}
\end{equation}
is a.s. finite and it is readily checked to be a backward
coalescence time. Applying Proposition \ref{confetto} the existence
and uniqueness of a stationary alternating renewal process is
proved. Condition (\ref{unifrate}) is certainly more restrictive
than needed, since it implies that the distributions of the holding
times have exponential tail bounds, whereas it is enough that they
have finite mean, see \cite{BB}.

Next we check the sufficient conditions in Proposition \ref{th0}.
Since
$$
a_k=1-\max(\sup_{h \geq k}p_h(-1,-1)-\inf_{h \geq
k}p_h(-1,-1),\sup_{h \geq k}p_h(1,1)-\inf_{h \geq k}p_h(1,1)), k \in
\mathbf{N}_+,
$$
it follows that $a_k \to 1$, as $k \to \infty$, if and only if
$p_k(i,i)$ converge as $k \to \infty$, for $i=-1,1$. If this
condition fails Proposition \ref{th0} cannot be applied. Even if
this condition holds, the speed of convergence of $a_k$ to $1$ can
be so slow that condition $(a)$ in Proposition \ref{th0} is still
violated. This happens, for example, for
$$
p_h(i,i)=\frac{1}{2}\left (1-\frac{1}{\sqrt {h+1}}\right ),
\,\,\,i\in \{-1,1\}, h \in \mathbf{N}_+\cup \{\infty\}.
$$
\end{example}

The previous  example suggests to investigate alternative ways to
define an information depth for the maximal coupling function, order
for the construction of a backward coalescence time under weaker
assumptions than those given in Proposition \ref{th0}.

\medskip
Let us introduce the process $\{A_h,\,\,\, h \in \mathbf{N}\}$, with
$A_0=a_0$ and
\begin{equation} \label{Agrande}
A_{h } ( U_{-1}, \dots , U_{-h} ) : =\inf \left\{ a_h (
\mathbf{w}^{-1}_{-h} ) :\mathbf{w}^{-1}_{-\infty} \in
J_h(U_{-1},\ldots,U_{-h})\right
 \}, \,\,\,  h \in \mathbf{N}_+,
\end{equation}
where
\begin{equation}  \label{randA}
J_h(U_{-1},\ldots,U_{-h})=\left\{ \mathbf{w}^{-1}_{-\infty}\in
\mathcal{H} : w_{-k} = g ,\hbox{ if } U_{-k} \in B_0(g),\,\,\, g \in
G,\,\, k\leq h \right \}
\end{equation}
is a set of histories smaller than $\mathcal{H}$, since the states
which can be identified, given $U_{-1},\ldots,U_{-h}$, are kept
fixed. In fact, whenever $U_{-k} \in B_0 (g)$, the function
$f(U_{-k};\mathbf{w}^{-(k+1)}_{-\infty})$ is equal to $g$,
irrespectively of the previous history
$\mathbf{w}^{-(k+1)}_{-\infty} \in \mathcal{H}$. Since $A_{h } (
U_{-1}, \dots , U_{-h} )$ is an infimum computed on a smaller set,
it holds
\begin{equation}\label{bbb}
A_{h } ( U_{-1}, \dots , U_{-h} )\geq a_h, h\in \mathbf{N},
\end{equation}
for any realization of the i.i.d. sample $\mathbf{U}^{-1}_{-h}$.

\comment{ ==========Now notice that the sequence
$\mathbf{V}^{-1}_{-\infty}$ with components
$$
V_{-k}=  \left \{  \begin{array}{ll}
         g & \hbox{ if } U_{-k} \in B_0 (g), \,\,\, g  \in G \\
         \partial & \hbox{ otherwise }     \\
       \end{array}   \right .
$$
is a generalized Bernoulli scheme with values in $\tilde G=G\cup
\{\partial\}$. If $V_{-k}=g \in G$, then
$f(U_{-k};\mathbf{w}^{-(k+1)}_{-\infty})=g$, irrespectively of the
previous history $\mathbf{w}^{-(k+1)}_{-\infty}$. Therefore
$$
A_{h} ( U_{-1}, \dots , U_{-h} )=\psi_h(V_{-1}, \dots , V_{-h} ),
$$
for some function $\psi_h: \tilde G^h \rightarrow \mathbf{R}^+$, for
$h\in \mathbf {N}_+$. ========================}

Now let us define
\begin{equation}  \label{tempi}
K'(\mathbf{U}^{0}_{-\infty})=\inf\{j \in \mathbf{N}: U_{0}<A_{j} (
U_{-1}, \dots , U_{-j} )\}.
\end{equation}

\begin{proposition}\label{infod}
    If $\lim_h A_h(U_{-1},\ldots,U_{-h})=1$ a.s., $K'$ is an
    information depth for the maximal coupling function.
\end{proposition}
\begin{proof}\label{infoprova}
 Only property (\ref{Kappa}) needs some explanations.
Suppose that $ K'(\mathbf{U}^{0}_{-\infty}) =m  $. Then
$$
 U_0 < A_{m} ( U_{-1}, \dots , U_{-m} ) \leq a_m (
\mathbf{w}^{-1}_{-m}), \,\,\, \forall  \mathbf{w}^{-1}_{-\infty} \in
J_m(U_{-1}, \dots,U_{-m} ),
$$
which from the definition of the maximal coupling function implies
\begin{equation}\label{serva}
  f(U_0 |\mathbf{w}^{-1}_{-\infty} ) = f(U_0 |\mathbf{z}^{-1}_{-\infty}
) ,\,\,\, \hbox{ for } \mathbf{w}^{-1}_{-\infty},
\mathbf{z}^{-1}_{-\infty} \in J_m(U_{-1}, \dots,U_{-m} ) , \hbox{
such that } \mathbf{w}^{-1}_{-m} =\mathbf{z}^{-1}_{-m}.
\end{equation}
Finally consider any $ \mathbf{w}^{-(m +1)}_{-\infty} \in
\mathcal{H}$; then, by choosing
$$
 \mathbf{w}^{-1}_{-m} = ( f^{(m)} (U_{-1} , \ldots,U_{-m} |  \mathbf{w}^{-(m +1)}_{-\infty} ),
 \ldots,  f^{(1)} (U_{-m} |  \mathbf{w}^{-(m +1)}_{-\infty} )),
$$
the history $\mathbf{w}^{-1}_{-\infty} \in J_m(U_{-1} ,
\ldots,U_{-m} )$. Now set $\mathbf{z}^{-1}_{-\infty} = (
\mathbf{w}^{-1}_{-m} , \mathbf{g}^{-(m +1)}_{-\infty} )$. If $
\mathbf{z}^{-1}_{-\infty} \in \mathcal{H}$ then it belongs
necessarily to $ J_m(U_{-1}, \dots,U_{-m} ) $ , in which case
 formula (\ref{serva}) is turned into (\ref{Kappa}).

\end{proof}

Next define $\{K'_{j}=K'(\mathbf{U}^{j}_{-\infty}), j \in
\mathbf{Z}\}$, and $\tau_0^{\mathbf{K}'}$ and $R_0^{\mathbf{K}'}$ as
in (\ref{ruo2}) and (\ref{taucappa}), respectively: then the
following result holds.

\begin{theorem}
\label{rinnov} If  $  \prod_{h=0}^{\infty} A_h (U_{h-1} , \dots ,
U_{0} )^{-1} \in  \mathcal{L}^1$ then $P(R_0^{\mathbf{K}'}) >0 $.
\end{theorem}

\begin{proof}\label{jk}
The sequence
\begin{equation}\label{marta}
Y_n := \frac{\prod_{h =0}^{n} \mathbf{1}_{\{ U_h < A_h (U_{h-1} ,
\ldots , U_0 ) \}} }{\prod_{h =0}^{n}A_h (U_{h-1} , \ldots , U_0 )
},\,\,\,n \in \mathbf{N},
\end{equation}
is uniformly integrable, because it is bounded by the integrable
random variable $ {\prod_{h=0}^{\infty} A_h (U_{h-1} , \dots , U_{0}
)^{-1}}$. Moreover it is a martingale with respect to the filtration
$\{ \mathcal{F}^n_0=\sigma ( U_{n} , \ldots, U_0 ):n \in \mathbf{N}
\}$. In fact, since $ U_n $ is independent of $\mathcal{F}^{n-1}_0
$, it holds
$$
E\left ( \mathbf{1}_{\{ U_n < A_n (U_{n-1} , \ldots , U_0 ) \}} |
\mathcal{F}^{n-1}_0 \right )=A_n (U_{n-1} , \ldots , U_0 ),
$$
hence
$$
E( Y_n | \mathcal{F}^{n-1}_0)= \frac{\prod_{h =0}^{n-1}
\mathbf{1}_{\{ U_h < A_h (U_{h-1} , \ldots , U_0 ) \}} }{\prod_{h
=0}^{n}A_h (U_{h-1} , \ldots , U_0 ) }E\left ( \mathbf{1}_{\{ U_n <
A_n (U_{n-1} , \ldots , U_0 ) \}} | \mathcal{F}^{n-1}_0 \right
)=Y_{n-1}.
$$
Furthermore, since $ Y_0 ={a_0 }^{-1} { \mathbf{1}_{\{ U_0 < a_0 \}}
}$, it follows that $E (Y_n ) =E(Y_0 ) =1$.

Since $\{Y_n\} $ is uniformly integrable, from a well known result
(see \cite{Wi}, page 134), the limit $ Y_\infty :=\lim_{ n \to
\infty } Y_n $ is finite a.s. with $ E(Y_\infty )=1 $. As a
consequence $Y_\infty
>0 $ with positive probability. But clearly
\begin{equation}\label{aprile}
  \{  Y_\infty >0 \}=\bigcap_{ h=0}^{\infty} \{  U_h < A_h (U_{h-1} , \ldots , U_0
  )\}=\{K'_h\leq h, h \in \mathbf{N}\}=R_0^{\mathbf{K}'},
\end{equation}
which ends the proof.
\end{proof}

By applying Proposition \ref{rinnovo}, the previous theorem implies
that a perfect simulation algorithm can be constructed from the
information depths $\{K'_{j},j \in \mathbf{Z}\}$.

{ {\bf Example 1 } (continued). \label{esempio} \it We prove that
Theorem  \ref{rinnov} can be applied to Example \ref{laburina},
under the assumption
\begin{equation}\label{motivo}
 a_0 (-1) >0 , \,\,\,  a_0 (1) >0 , \,\,\,  a_\infty := \sup_{k \geq 0 } a_k  > 1 -2 a_0 (-1)a_0 (1) .
\end{equation}
Notice that the latter condition is automatically verified if the
former ones hold and $a_\infty =1$.
  Define
\begin{equation}\label{stopper2}
N(\mathbf{U}_{0}^{\infty})=\inf \{n\in \mathbf{N}:
1_{[a_0(-1),a_0)}(U_n)1_{[0,a_0(-1))}(U_{n-1})+1_{[a_0(-1),a_0)}(U_{n-1})1_{[0,a_0(-1))}(U_{n})=1\}.
\end{equation}
Then if $h>N(\mathbf{U}_{0}^{\infty})$, and
$\mathbf{w}^{h-1}_{-\infty} \in J_h(U_{h-1},\ldots,U_0)$ then
$$
a_h(g|\mathbf{w}^{h-1}_0)=p(g|\mathbf{w}^{h-1}_{-\infty}).
$$
In this case
$$
A_h (U_{h-1} , \dots , U_{0} )=\inf \left\{\sum_{g \in
G}a_h(g|\mathbf{w}^{h-1}_0): \mathbf{w}^{h-1}_{-\infty} \in
J_h(U_{h-1},\ldots,U_0)\right\}
$$
$$
=\inf \left\{\sum_{g \in G}p(g|\mathbf{w}^{h-1}_{-\infty}):
\mathbf{w}^{h-1}_{-\infty} \in J_h(U_{h-1},\ldots,U_0)\right\}=1.
$$
Therefore, for any $\delta\in (0 , a_\infty - 1 + 2 a_0(-1)a_0(1))$,
choosing $n_0=n_0(\delta)$ such that $a_{n_0}\geq 1-2
a_0(-1)a_0(1)+\delta$, it is obtained
\begin{equation}\label{integr}
(\prod_{h=0}^{\infty} A_h (U_{h-1} , \dots , U_{0}
))^{-1}=(\prod_{h=0}^{N} A_h (U_{h-1} , \dots , U_{0} ))^{-1}\leq (
1-2 a_0(-1)a_0(1)+\delta  )^{-N} \prod_{h=0}^{n_0}a_h^{-1}.
\end{equation}
The expression (\ref{stopper2}) suggests a majorization of $N$ with
twice a geometric random variable having the success probability
$p=2a_0(1)a_0(-1)$. Since the radius of convergence for the p.g.f.
of this kind of random variable is $1/ (1-p)$ the assumption
(\ref{motivo}) and the bound (\ref{integr}) imply that
$\prod_{h=0}^{\infty} A_h (U_{h-1} , \dots , U_{0} ))^{-1}$ is
integrable. Therefore Theorem \ref{rinnov} can be applied, showing
that the stationary alternating renewal process can be perfectly
simulated. }

\comment{ for $h \geq 2$,
\begin{equation*}
C_h = \{ (u_{h-1} , \ldots , u_{0}) : \exists 0 \leq s \leq h-2 ,
\,\,\, u_s < c(-1, \infty) , \hbox{ } c(-1, \infty) < u_{s +1}
<2c(-1, \infty)\} .
\end{equation*}
It is clear that the random variable $N $ is dominated by a
geometric r.v. hence assumption 1) is satisfied. Since $\alpha_h =1$
for $h \geq 2$ it is $\prod_{h=2}^{\infty } \alpha_{h } =1$, hence
assumption 2) is satisfied as well. Since 3) was already assumed in
Section \ref{sec2}, we conclude that the events $ \{C_h: h \geq 2\}$
satisfy the assumption of Proposition \ref{teo2}.

Independently for each site $ i \in \mathbb{Z}$ we construct a
random variable $ U_i $ uniform in $[0,1]$ by a two stage procedure.
First draw $Z_i $ with values in $\tilde G$ such that
\begin{equation}\label{Z}
    P(Z_i =g ) = |B_0 (g)| ,\,\,\,  g \in G
\end{equation}
\begin{equation}\label{Z2}
    P(Z_i =0 ) = 1- |B_0 | .
\end{equation}
Next choose $U_i $ to be uniform in $ B_0(g )$ if $Z_i =g$ and $U_i
$ uniform in $[0,1] \setminus B_0 $ if $Z_i =0$. Since $ A_h
(U_{h-1} , \dots ,    U_{0} )=f_h (Z_{h-1} , \dots ,    Z_{0} )$ we
define the event $   C\subset (\tilde G)^{-\mathbb{N}_+} $ by
\begin{equation}\label{insieme}
    C =  \{ z =(z_0 , z_1 , \ldots) : \prod_{h=0
    }^{\infty} f_h (z_{h-1} , \ldots ,    z_{0} ) >0\} .
\end{equation}

?????????????????????????????

\end{proof}

}

Inspired by the previous example, in the following corollary we
present a sufficient condition, possibly easier to verify, which
guarantees that the assumption in Theorem \ref{rinnov} holds.
\begin{corollary}
\label{teo2} Let $C_h$ be a Borel subset of $[0,1)^h$ and suppose:
\begin{itemize}
    \item[1)] the sequence \begin{equation}  \label{alfah}
\alpha_h=\inf\{A_h(u_{h-1},...,u_{0}):(u_{h-1},...,u_{0})\in C_h\}
\end{equation}
is such that $\prod_{h=0}^{\infty} \alpha_{h } >0 $ (in particular
$\alpha_0=a_0>0$);
    \item[2)] the random variable
\begin{equation}\label{tempo2}
    N:=N(\mathbf{U}^{+\infty}_{0}) = \inf \{ m : (U_{n-1} , \ldots , U_0) \in C_n, \forall n \geq m\}
\end{equation}
has a probability generating function $E ( s^{ N} )< \infty $ for
some $ s
> 1/a_\infty$.
\end{itemize}
Then
\begin{equation}  \label{pal}
E\left ( \prod_{h=0}^{\infty } A_{h }(U_{h-1},...,U_{0})^{-1} \right
)<\infty .
\end{equation}
\end{corollary}

\begin{proof} \label{2consuf}
By definition of $N$, $ A_h \geq \alpha _h $ for $h > N $. Moreover,
since $A_h  \geq a_h $ for each integer $h$, we have
\begin{equation}\label{disA}
    \frac{1}{\prod_{h=0}^{\infty} A_h }\leq  \frac{1}{\prod_{h=0}^{N} a_h
    }   \frac{1}{\prod_{h=N+1}^{\infty} \alpha_h } .
\end{equation}
By taking expected values at both sides, with a straightforward
bound for the second factor at the r.h.s., it is obtained
\begin{equation}\label{disA2}
    E \left (\frac{1}{\prod_{h=0}^{\infty} A_h } \right) \leq  E  \left (\frac{1}{\prod_{h=0}^{N} a_h
    } \right) \frac{1}{\prod_{h=0}^{\infty} \alpha_h }
     .
\end{equation}
The second factor at the r.h.s. is finite by assumption 1). By
assumption there exists an integer $k $ such that $ 1/a_k < s$, $s$
being as in 2). By consequence we have the following bound for the
first factor
\begin{equation}\label{bbq}
  E  \left (\frac{1}{\prod_{h=0}^{N} a_h
    } \right) \leq \frac{1}{\prod_{h =0}^{k-1}    a_h    }E( s^{ N}
    )< + \infty,
\end{equation}
from which the corollary follows.
\end{proof}

Finally we provide another class of models that satisfy the
conditions of Corollary \ref{teo2} but not those of Proposition
\ref{th0}.

\begin{example}\label{pezzoforte}
Consider positive summable sequences $\beta(i)\geq \gamma(i),\,\, i
\in \mathbf{N}_+$,
and assume that $p_1\in (0,1)$, $\sigma
>0 $ and $c
>0 $ are such that
\begin{equation}\label{did}
    p_1 (  1- c\sum_{i=1}^{\infty} \beta (i)) > \sigma;
\end{equation}
moreover assume that $\sum_{i=1}^{\infty}i\gamma(i)<\infty$.

Now define the kernel $p$ on $G=\{0,1\}$ by
\begin{equation}\label{specific}
    p(1 | \mathbf{w}^{-1}_{-\infty})= p_1 \{1 - c \sum_{i=1}^{\infty}(\beta (i) \mathbf{1}_{\{ w_{-i} =0, T (  \mathbf{w}^{-1}_{-\infty} )> i\}}
    +\gamma (i) \mathbf{1}_{\{ w_{-i} =0, T (  \mathbf{w}^{-1}_{-\infty} )\leq i \}}  )  \},
\end{equation}
where
\begin{equation}\label{arresto}
    T (  \mathbf{w}^{-1}_{-\infty} ) = \inf \left \{k : \frac{ \sum_{i=1}^{k} w_{-i} }{ k } \geq \sigma \right
    \}.
\end{equation}
First of all we prove that the kernel $p$ is monotone, which means
that $p(1 | \mathbf{w}^{-1}_{-\infty})$ is increasing in
$\mathbf{w}^{-1}_{-\infty}$ w.r.t. the pointwise order. For this
notice that $w_{-i}\geq \eta_{-i}$, for $i \in \mathbf{N}_+$ implies
$T(\mathbf{w}^{-1}_{-\infty})\leq T(\mathbf{\eta}^{-1}_{-\infty})$,
hence
$$
p(1 | \mathbf{w}^{-1}_{-\infty})\geq p_1 \{1 - c
\sum_{i=1}^{\infty}( \beta (i) \mathbf{1}_{\{ {\eta}_{-i} =0,
T(\mathbf{w}^{-1}_{-\infty})>i \}}
    +\gamma (i) \mathbf{1}_{\{ {\eta}_{-i} =0, T(\mathbf{w}^{-1}_{-\infty})\leq i \}}  )
$$
$$
\geq p_1 \{1 - c \sum_{i=1}^{\infty}( \beta (i)
\mathbf{1}_{\{{\eta}_{-i} =0, T({\mathbf{\eta}}^{-1}_{-\infty})>i
\}}
  +\gamma (i) \mathbf{1}_{\{ {\eta}_{-i} =0, T({\mathbf{\eta}}^{-1}_{-\infty})\leq i \}}
    )= p(1 | {\mathbf{\eta}}^{-1}_{-\infty}),
$$
where the second inequality is due to the fact that $ \beta (i )
\geq \gamma (i ) $, for $i \in \mathbf{N}_+$.

Since $a_0(1)=p_1(1-c\sum_{i=1}^\infty \beta (i))>0$ and
$a_0(0)=1-p_1>0$ it follows that $\mathcal{H}=G^{-\mathbf{N}_+}$. As
a consequence
\begin{equation}\label{asdf2}
  a_k (0 , \mathbf{w}^{-1}_{-k})  = \inf \{ p (0| \mathbf{w}^{-1}_{-k},
    \mathbf{z}^{-k-1}_{-\infty}) : \mathbf{z}^{-k-1}_{-\infty} \in \{0,1\}^{-\mathbf{N}_+}\}
   = p(0 | \mathbf{w}^{-1}_{-k}, \mathbf{1}^{-k-1}_{-\infty}  ) 
\end{equation}
and
\begin{equation}\label{asdf}
  a_k (1, \mathbf{w}^{-1}_{-k})  = \inf \{ p (1| \mathbf{w}^{-1}_{-k},
    \mathbf{z}^{-k-1}_{-\infty}) : \mathbf{z}^{-k-1}_{-\infty} \in \{0,1\}^{-\mathbf{N}_+}\}
   = p(1 | \mathbf{w}^{-1}_{-k}, \mathbf{0}^{-k-1}_{-\infty} )
\end{equation}
therefore
\begin{equation}\label{akw}
a_k ( \mathbf{w}^{-1}_{-k}) =p(0 | \mathbf{w}^{-1}_{-k},
\mathbf{1}^{-k-1}_{-\infty} ) +p(1 | \mathbf{w}^{-1}_{-k},
\mathbf{0}^{-k-1}_{-\infty} )=1-p (1| \mathbf{w}^{-1}_{-k},
\mathbf{1}^{-k-1}_{-\infty})+p(1 | \mathbf{w}^{-1}_{-k},
\mathbf{0}^{-k-1}_{-\infty} )
\end{equation}
which by a direct computation is seen to assume only the values $ 1
- p_1 c \sum_{i=k+1}^\infty \gamma (i)$, when $T
(\mathbf{w}^{-1}_{-\infty} ) \leq k$, and $ 1 - p_1 c
\sum_{i=k+1}^\infty \beta (i)$, otherwise. Notice that the condition
$T (\mathbf{w}^{-1}_{-\infty} ) \leq k$ can be verified by looking
only at $\mathbf{w}^{-1}_{-k}$.

Now we prove that this class of kernels can be perfectly simulated.
In fact we can prove that the conditions given in Corollary
\ref{teo2} are satisfied for the sequence of events
 \begin{equation}\label{Cn}
    C_n = \left \{
    (u_{n-1} , \ldots, u_0 )\in [0,1)^n : \frac{1 }{n}
     \sum_{k=0}^{n-1} \mathbf{1}_{\{ a_0 (0) \leq      u_k < a_0 (0)
    + a_0 (1)\}} \geq \sigma.
    \right \}
\end{equation}
Since $a_0 (1)> \sigma $, from Chernoff's bound
$$
P( (U_{n-1} , \ldots, U_0 )\notin C_n ) \leq e^{-K n},
$$
for some $K>0$, therefore
$$
P(N \leq n_0) = P ( \cap_{n=n_0}^{\infty} \{  (U_{n-1} , \ldots, U_0
)\in C_n \} ) \geq 1 -  \sum_{n=n_0}^{\infty} e^{-K n}= 1 -
\frac{e^{-Kn_0}}{ 1-e^{-K } }
$$
from which the existence of the probability generating function of
$N$, for some $s>1$, is deduced,. From (\ref{akinf}) and the
summability of $\beta (i )$, it is obtained that $a_{\infty}=1$,
which ensures that condition 2) of Corollary \ref{teo2} is
satisfied. Finally observe that whenever $ (U_{n-1} , \ldots, U_0
)\in C_n $
$$
A_n (U_{n-1} , \ldots, U_0 ) \geq \inf \{ a_n (0 |
\mathbf{w}^{-1}_{-n})+a_n (1 | \mathbf{w}^{-1}_{-n}) : T
(\mathbf{w}^{-1}_{-\infty} ) \leq n  \} = 1 - p_1 c \sum_{i=n
+1}^{\infty} \gamma (i ) ,
$$
therefore, by definition (\ref{alfah}), we get $ \alpha_n   \geq 1 -
p_1 c \sum_{i=n +1}^{\infty} \gamma (i )$. Now, being
$$
\sum_{n=0}^{\infty}\sum_{i=n+1}^{\infty}  \gamma (i
)=\sum_{i=1}^{\infty} i \gamma (i)< \infty
$$
we get by \cite{Wi}, page 40, that $\prod_{n=1}^{\infty}\alpha_n>0$
, so that condition 1) is also satisfied.

On the other hand, for some choices of $ \{ \beta (i)\} $ and $\{
\gamma (i)\}$, condition a) in Proposition \ref{th0} fails. For
example consider $\beta(i)=i^{-\alpha}$, with $\alpha \in (1,2)$ and
$\gamma (i)= 2^{-i}$, which ensure that $\beta(i) \geq \gamma (i)$,
for $i \in \mathbf {N}_+$, and $\sum_{i=1}^{\infty} i \gamma (i) <
\infty$. Since
\begin{equation}\label{akinf}
    a_k =  \inf \{ a_k ( \mathbf{w}^{-1}_{-k}) :  \mathbf{w}^{-1}_{-k}\in \{0,1\}^k \}  =  1 - p_1 c
\sum_{i=k+1}^\infty \beta (i) ,
\end{equation}
we can show that $ \sum_{k=1}^\infty  \prod_{i=1}^k  a_i < \infty $.
In fact
$$
\sum_{k=1}^\infty  \prod_{i=1}^k  a_i = \sum_{k=1}^\infty
\prod_{i=1}^k  \left (1 - p_1 c \sum_{i=k+1}^{\infty}
\frac{1}{i^\alpha} \right ) \leq \sum_{k=1}^\infty \prod_{i=1}^k
\left (1 - \frac{L_1}{i^{\alpha -1}} \right ) ,
$$
where $L_1 >0$ is a sufficiently small constant. The rightmost
expression is smaller than
$$
\sum_{k=1}^\infty \exp ({   \sum_{i=1}^k  -\frac{L_1}{i^{\alpha -1}}
} ) \leq C\sum_{k=1}^\infty \exp (-L_2 k^{2-\alpha}) <\infty
$$
where $C$ and $L_2$ are suitable positive constants, which implies
the promised inequality.

\end{example}

We conclude the section by observing that the idea of defining the
information depth by computing the infimum of $a_h(
\mathbf{w}^{-1}_{-h} )$ over the set $J_h(U_{-1}, \ldots,U_{-h} )$
of histories compatible with the observed $U_{-1}, \ldots,U_{-h} $
can be pushed further. For example, by looking at adjacent pairs  $
(U_{-i} , U_{-i+1} )$, $i = 2, \ldots , h $, it is possible to
locate other states, restricting the set of histories compatible
with the observed $U_{-1}, \ldots,U_{-h} $ to the smaller subset
$$
 J'_h  (U_{-1}, \ldots,U_{-h} )= J_h  (U_{-1}, \ldots,U_{-h} )
                         \cap F_h  (U_{-1}, \ldots,U_{-h} ) ,
$$
where $  F_h  (U_{-1}, \ldots,U_{-h} ) $ is equal to
$$
\left \{    \mathbf{w}^{-1}_{-\infty}\in \mathcal{H} :
    w_{-k} = g_1 ,   w_{-k+1} = g_2,  \hbox{ if }
U_{-k} \in B_0(g_1),  \,\,\,       U_{-k+1} \in B_1(g_2| g_1), g_1,
g_2 \in G,\,\,        2 \leq   k\leq h \right \} .
$$
The changes to Proposition \ref{infod} and Theorem \ref{rinnov} are
minor, but for the sake of brevity, we do not pursue this extension
further.

\section{An algorithm which works without minorization condition} \label{sec4}

In this section we explore the possibility of defining a backward
coalescence time $\tau_0$ when $a_0 =0$. In this case any
information depth takes necessarily positive values, hence it cannot
be used for defining a backward coalescence time. However it is
assumed $a_1 >0$.
Since $a_1 \leq a_1 ( w_{-1}) =\sum_{ g \in G } a_1(g | w_{-1} ) $
for any $w_{-1} \in G $, the maximal coupling function $f (u,
\mathbf{w}^{-1}_{-\infty})$ depends only on $w_{-1}$, whenever $u <
a_1$. Accordingly, we say that the simulation process is in the {\it
markovian regime} at time $n$ whenever $U_n<a_1$. This means that
the information needed to compute the state of the process at time
$n$ concerns only the state at time $n-1$. For any $ u \in [0,1)$
and $w \in G $ we define
\begin{equation}  \label{markov}
\tilde{f} \left (u| w \right ):= f (a_1u |
\mathbf{w}^{-1}_{-\infty}),
\end{equation}
for any choice of $\mathbf{w}^{-1}_{-\infty} \in \mathcal{H}$ having
$ w_{-1} =w$. Thus $\tilde{f}: [ 0, 1) \times G \to G $, applied to
a uniform random variable in $[ 0, 1)$, induces the Markov kernel
\begin{equation}\label{kern}
M(g|w)=|\{u \in [0,1):\tilde {f}(u|w)=g\}|,\,\,\, g,w \in G,
\end{equation}
where $| \cdot |$ denote the Lebesgue measure.

By induction, for $n\geq 2 $, we define the composition
$\tilde{f}^{(n)}: [ 0, 1)^n \times G \to G $ as
\begin{equation}  \label{comp}
\tilde{f}^{(n)} ( u_n , \dots, u_1|w ) =\tilde{f} ( u_n | \tilde{f}%
^{(n-1)} ( u_{n-1} , \dots, u_1 | w) ), \,\,\, u_i \in [0,1), \,\,\,
i =1, \dots, n , \,\,\, w \in G
\end{equation}
where $\tilde{f}^{(1)} = \tilde{f}$.

Concerning the markovian regime, for any $n \in \mathbf{N}_+ $, we
define the {\it coalescence in the interval $[-n+1,0 ]$} as
\begin{equation}  \label{coup}
E _n = \{(u_0,u_{-1},\ldots,u_{-n+1})\in [0,1)^n :\, \tilde{f}^{(n)}
( u_0 , \dots, u_{-n+1} | w ) = \tilde{f}^{(n)} ( u_0 , \dots,
u_{-n+1} | g_0) , \forall w \in G \},
\end{equation}
where $g_0 \in G$ is an arbitrary state.


We notice that if the kernel $p$ is markovian, then $a_1=1$ and
conversely. In this case $\tilde{f}=f$ and any backward coalescence
time has the property that $\tau_0=-m$ implies that
$(U_0,\ldots,U_{-m+1}) \in E_m$, as in the original CFTP algorithm
\cite{PW}.


Next assume that $a_k \uparrow 1 $ as $k \to \infty$ and recall that
in this case $K$, as  defined in (\ref{kappa}), takes finite values.
Define the random variable $ {\tau }_0$ as
\begin{equation}  \label{b}
\sup\{m<0: \exists l\in [m,0], \hbox{ s.t. }
a_1^{-1}(U_l,\ldots,U_{m})\in E_{l-m+1}, \,\,\, 
\& \,\,\, K(U_j) \leq j-l, j\in [l+1,0 ] \hbox{ if } l <0 \} .
\end{equation}
\begin{proposition}
If the random variable $\tau_0$ is finite almost surely, it is a
backward coalescence time.
\end{proposition}
\begin{proof}
By definition,  for any $m \in -\mathbf{N}$, the event ${\tau }_0=m$
belongs to the $\sigma$-algebra $\mathcal {F}^{0}_{m}$, which proves
H1. Moreover, if this event is realized the process is in the
markovian regime from time $m$ to some larger time $l$ in which
coalescence has taken place. This means that
$$
\tilde f^{(l-m+1)}\left(\frac {U_l}{a_1},\ldots,\frac
{U_m}{a_1}|w\right) =\tilde f^{(l-m+1)}\left(\frac
{U_l}{a_1},\ldots,\frac {U_m}{a_1}|g_0\right)
$$
for any $w \in G$. By the relation (\ref{markov}) this means that,
for any $\mathbf{w}^{-1}_{-\infty} \in \mathcal{H}$, it holds
\begin{equation}\label{secondopasso}
f^{(l-m+1)}(U_l,\ldots,U_m|\mathbf{w}^{-(m+1)}_{-\infty}) =
f^{(l-m+1)}(U_l,\ldots,U_m|\mathbf{g}^{-(m+1)}_{-\infty})
\end{equation}
Thus, if $l=0 $, H2 holds. If $l <0$ one needs to repeat the proof
of Proposition \ref{acca} replacing (\ref{primopasso}) with
(\ref{secondopasso}). In short, to compute all the states of the
process in the interval $[l+1,0]$, there is no requirement about the
states of the process prior to time $l$.
\end{proof}

After this result, we turn our interest to give sufficient
conditions for the a.s. finiteness of $\tau_0$.

\begin{theorem}
\label{co} Under the assumptions
\begin{itemize}
    \item [(i)] $\sum_{n=1}^{\infty} \prod_{m=1}^{n} a_m = \infty $;
    \item [(ii)] there exists $s\in \mathbf{N}_+$ such
that $ P((U_{s-1},\ldots,U_0) \in E_s)>0$;
\end{itemize}
${\tau }_0$ is finite a.s.
\end{theorem}

\begin{proof}
We start by defining the sequences $ \{W_n , n = 1,2 \dots \}$, $
\{Y_n , n = 1,2, \dots \}$ which will be proved to be finite a. s.
First define
$$
W_1=\sup\{m \leq 0 :K(U_j)-1\leq j-m, j\in [m,0]\}.
$$
By Proposition \ref{th0}, part (a), condition (i) guarantees that
$W_1$ is a.s. finite: notice indeed that replacing $K(U_j)$ with
$K(U_j)-1$ has the effect of shifting the sequence $\{a_j, j \in
\mathbf{N}\}$ to the left. Next define
$$
Y_i = \inf \{m < W_{i} : U_n <a_1, n\in[m+1,W_{i}] \}  , \,\,\,\,
  W_{i+1}= \sup\{m\leq Y_i:K(U_j)-1\leq j-m, j\in [m,Y_i]\},
$$
which are a.s. finite, for $ i\in \mathbf{N}_+$. It is immediately
seen that $\{W_i - Y_i -1\}_{i\in\mathbf{N}_+}$ is a sequence of
i.i.d. geometric random variables, with success probability $1-a_1$.
Likewise $\{W_{i+1} -Y_i\}_{i\in\mathbf{N}_+}$ is a sequence of
i.i.d. random variables distributed as $W_1$, conditional to be non
zero. Moreover the two sequences are mutually independent and
independent of $W_1$. In particular the sequence $\{-W_i, i \in
\mathbf{N}_+\}$ form a delayed renewal process and the sequence
$\{U_{-n}, n \in \mathbf{N}\}$ is regenerative w.r.t. it.

Finally define the random index
\begin{equation}\label{alal}
   Q=\inf \{i\in \mathbf{N}_+: (U_{W_i-1},\ldots,U_{Y_i+1}) \in
E_{W_i-Y_i-1}\}
\end{equation}
and let $\tau^*=Y_Q$.

For each $j \in [W_i , 0 ]$ the condition $ j - K(U_j ) \geq W_i -1
$ is satisfied, for any $i \in \mathbf{N}_+$. By consequence $
\tau^* $ differs from ${\tau }_0 $ only because the supremum is
taken on the set $ \{Y_i:{i \in \mathbf{N}_+} \}$ rather than on the
whole set of negative integers. In fact notice that, for $n
=\tau^*=Y_Q $, one can always choose in (\ref{b}) $ l = W_Q -1$.
Therefore $\tau^* \leq {\tau }_0$, so it is enough to prove that
$\tau^*> - \infty$ a.s. But this is true because, by assumption
(ii), the condition at the r.h.s. of (\ref{alal}) is fulfilled with
positive probability in any regenerating cycle: an application of
the law of large numbers concludes the proof.

\end{proof}

In the previous theorem we have not assumed that $a_0=0$. However,
in this case the result does not add anything to the statement (a)
in Proposition \ref{th0}. Indeed, if $U_0$ is uniformly distributed
in $[0,1)$, $P(U_0 \in E_1)>0$ if $a_0>0$: assumption (ii) is always
satisfied. Therefore, in the following we will always take $a_0=0$.

Assumption (ii) in the previous theorem states that the markovian
coupling function $\tilde {f}$ defined in (\ref{markov}) is {\it
successful} for the perfect simulation of the Markov chain with
kernel $M$ given in (\ref{kern}), in the sense that backward
coalescence occurs with probability $1$. Since this implies the
convergence in law of the chain as time increases, it is necessary
that $M$ has a single positive recurrent irreducible class which is
aperiodic.

When $G$ is finite, which is assumed from now on, this condition can
be directly referred to the oriented graph induced by $M$. Notice
that if $ (w,g)$ is an arc of this graph then necessarily $a_1 (g |
w)>0$. However the converse is not true. In fact if $ B_1(g| w) $ is
non empty and it is disjoint from $[0, a_1)$ then $(w , g ) $ is not
an arc. In this case, if $\mathbf{w}^{-1}_{-\infty}$ is such that
$w_{-1} =w$, the maximal coupling function can be replaced by a new
coupling function $ \bar f(u|\mathbf{w}^{-1}_{-\infty})$ which is
different only for $u <a_1 (w)$. Each  interval $B_1 (h| w)$ is
replaced by the union of two disjoint intervals $B^{1}_1(h|w)$ and
$B^{2}_1(h|w)$, where  $ \bar f(\cdot |\mathbf{w}^{-1}_{-\infty})$
takes the value $h$.

We require that
\begin{equation}\label{arrange}
|B^{1}_1(h|w)|+|B^{2}_1(h|w)|=
|B_1(h|w)|=a_1(h|w)=|\{u<a_1(w):f(u|\mathbf{w}^{-1}_{-\infty})=h\}|
\end{equation}
and $B^{1}_1(h|w)$ intersects the interval $[0,a_1)$. Therefore the
Markov kernel $\bar M$ induced by $\bar f$ satisfies
\begin{equation}\label{graph}
a_1(g|w)>0  \Leftrightarrow \bar {M}(g|w)>0.
\end{equation}

However, backward coalescence w.p. $1$ cannot be ensured only by
properties of the Markov kernel, without reference to the coupling
function. A simple counterexample is presented in \cite{Hag}. But
when the state space is finite, there is a universal modification of
a Markov coupling function, which ensures backward coalescence w.p.
$1$ under the only assumption that the induced kernel has a unique
irreducible class which is aperiodic (see Proposition 8.1 p. 122 in
\cite{AG}). The modification consists in letting the different
trajectories move independently before merging.

This is more clearly explained by allowing coupling functions to
depend on $n+1$ variables $ (u^0, \dots , u^n) \in [0,1)^{n+1}$,
rather than a single variable $u \in [0,1)$; in the definition just
replace intervals by hypercubes or more general Borel sets. The
modified coupling function $\bar f $ introduced before is replaced
by
 $\hat{f}:[0,1)^{|G|+1}\times \mathcal{H} \to G$, defined as
\begin{equation}\label{vector}
\hat{f}(u^0; u^g, g \in
G|\mathbf{w}^{-1}_{-\infty})=\left\{\begin{array}{cc}
\bar f(u^0|\mathbf{w}^{-1}_{-\infty}), & u^0\geq a_1, \\
\bar f(a_1 u^{w_{-1}} |\mathbf{w}^{-1}_{-\infty}), & u^0< a_1, \,\,\, w_{-1} \in G. \\
\end{array}\right .
\end{equation}
which is not difficult to check that remains a coupling function for
$p$.
As a corollary to Theorem \ref{co}, by collecting together the two
previous remarks, we can construct a backward coalescence time (and
thus a perfect simulation algorithm) for some interesting class of
kernels $p$. Notice that the last condition in the following
corollary has the purpose of ensuring  that $a_0=0$.

\begin{corollary}\label{randwalk} Suppose that $p: G \times
G^{-\mathbf{N}_+}\to [0,1]$ is a kernel on the finite state space
$G$. Define the oriented graph $\mathcal{G}$ with set of vertices
$G$ and the set of arcs $ \mathcal{A}=\{(w,g)\in G^2: a_1(g|w)>0\}$.
Suppose
\begin{itemize}
    \item [(i)] $\sum_{n=1}^{\infty} \prod_{m=1}^{n} a_m = \infty $;
    \item [(ii)] $\mathcal{G}$ has a single irreducible class which is
    aperiodic;
    \item [(iii)] for any $g \in G$, there exists $w \in G$ such that
    $(w,g) \notin \mathcal{A}$.
\end{itemize}
Then it is possible to construct a backward coalescence time for the
coupling function $\hat f$.
\end{corollary}

\begin{example} The previous result covers some \emph{generalized random walks} on a
finite directed graph $\mathcal{G}=(G,\mathcal{A})$. Before defining
this kind of processes, we define the set of one-sided infinite
paths in $\mathcal{G}$
$$
\mathcal{C}=\{\mathbf{w}^{-1}_{-\infty} \in G^{-\mathbf{N}_+}:
(w_{-(k+1)},w_{-k}) \in \mathcal{A}, k \in \mathbf{N}_+\}.
$$
Generalized random walks on $\mathcal{G}=(G,\mathcal{A})$ are
processes compatible with a kernel $p$ over the alphabet $G$ with
the properties:
\begin{itemize}
\item if $(g,w)\notin \mathcal{A}$ then, for all
    $\mathbf{w}^{-1}_{-\infty} \in G^{-\mathbf{N}_+}$ with
    $w_{-1}=w$, $p(g|\mathbf{w}^{-1}_{-\infty})=0$;
    \item if $(g,w)\in \mathcal{A}$, there exists $\epsilon>0$ s.t.
    for all $\mathbf{w}^{-1}_{-\infty} \in \mathcal{C}$
    with $w_{-1}=w$, $ p(g|\mathbf{w}^{-1}_{-\infty})>\epsilon$.
\end{itemize}
The first property implies that $\mathcal{H}\subset \mathcal{C}$,
whereas the second ensures the opposite inclusion. Moreover
$$
a_1(g|w)=\inf \left\{
p(g|\mathbf{w}^{-1}_{-\infty}):\mathbf{w}^{-1}_{-\infty} \in
\mathcal{C}, w_{-1}=w \right\}>\epsilon>0,
$$
if $(g,w)\in \mathcal{A}$ is an arc of $\mathcal{G}$, otherwise it
is clearly $a_1(g|w)=0$. Thus we can get the set $\mathcal{A}$ from
the kernel $p$ as indicated in Corollary \ref{randwalk}. Therefore
if the graph $\mathcal{G}$ satisfies conditions (ii) and (iii) and
the sequence $\{a_k, k \in \mathbf{N}_+\}$ satisfies condition (i),
the previous Corollary allows to prove the existence and uniqueness
of the generalized random walk, and the feasibility of a perfect
simulation algorithm for sampling it.
\end{example}

The result of this section can be extended to cover the case $ a_1=
\dots = a_l =0, a_{l+1}>0 $, for some $l \geq 1$. In this case the
maximal coupling function depends on at least $l+1$ variables hence
it induces a markovian kernel $M$ on the state space $G^{l+1}$. The
changes to the statement of Theorem \ref{co} are rather
straightforward.

\bigskip
\bibliographystyle{plain}

\end{document}